\documentclass{amsart}
\usepackage{textcomp}
\usepackage{indentfirst}
\usepackage{amsmath}
\usepackage{amsthm}
\usepackage{amsfonts}
\usepackage{amssymb}
\usepackage{mathrsfs}
\usepackage{stmaryrd}
\usepackage{bbm}
\usepackage[T1]{fontenc}
\usepackage{multirow}

\newtheorem{thm}{Theorem}[section]
\newtheorem{cor}[thm]{Corollary}
\newtheorem{lem}[thm]{Lemma}
\newtheorem{prop}[thm]{Proposition}
\theoremstyle{definition}

\newtheorem{exmp}[thm]{Example}

\numberwithin{equation}{section}

\begin{document}

\title[Asymmetric truncated Toeplitz operators II]{Asymmetric truncated Toeplitz operators on finite-dimensional spaces II}
\author[Bartosz \L anucha]{Bartosz \L anucha}
\address{Department of Mathematics, Maria Curie-Sk\l odowska University, Maria Curie-Sk\l odowska Square 1, 20-031 Lublin, Poland}%
\email{bartosz.lanucha@poczta.umcs.lublin.pl}%
\begin{abstract}
In this paper we present some consequences of the description of matrix representations of asymmetric truncated Toeplitz operators acting between finite-dimensional model spaces. In particular, we prove that these operators can be characterized using modified compressed shifts or the notion of shift invariance. We also describe rank-one asymmetric truncated Toeplitz operators acting between finite-dimensional model spaces. This paper is a sequel to the paper of J. Jurasik and B. \L anucha \textit{Asymmetric truncated Toeplitz operators on finite-dimensional spaces}.
\end{abstract}

\maketitle

\renewcommand{\thefootnote}{}



\footnote{\today}

\renewcommand{\thefootnote}{\arabic{footnote}}
\setcounter{footnote}{0}

\section{Introduction}

Let $H^2$ be the classical Hardy space of the unit disk
$\mathbb{D}=\{z\colon|z|<1\}$ and let $P$ denote the Szeg\"{o} projection, that is, the orthogonal projection from $L^2(\partial\mathbb{D})$ onto $H^2$.

The unilateral shift $S:H^2\rightarrow H^2$ is defined on $H^2$ by
$$Sf(z)=zf(z).$$
It is known that the $S$-invariant subspaces of $H^2$ can be characterized in term of inner functions. An inner function $\alpha$ is a function from $H^{\infty}$, the algebra of bounded analytic functions on $\mathbb{D}$, such that $|\alpha|=1$ a.e. on $\partial\mathbb{D}$. The theorem of A. Beurling states that every closed nontrivial $S$-invariant subspace of $H^2$ is of the form $\alpha H^2$ for some inner function $\alpha$. Consequently, each of the closed nontrivial subspaces of $H^2$ invariant under the backward shift $S^{*}$ is given by
$$K_{\alpha}=H^2\ominus \alpha H^2$$
 with $\alpha$ inner. The space $K_{\alpha}$ is called the model space corresponding to the inner function $\alpha$.

The model space $K_{\alpha}$ has a reproducing kernel property and for each $w\in\mathbb{D}$ the function
\begin{equation}\label{kerker}
k_{w}^{\alpha}(z)=\frac{1-\overline{\alpha(w)}\alpha(z)}{1-\overline{w}z},\quad z\in\mathbb{D},
\end{equation}
is the corresponding kernel function, that is, $f(w)=\langle f, k_{w}^{\alpha}\rangle$ for every
$f\in K_{\alpha}$ ($\langle
\cdot,\cdot\rangle$ being the usual integral inner product). Since $k_{w}^{\alpha}\in H^{\infty}$, the set $K_{\alpha}^{\infty}$ of all bounded functions in $K_{\alpha}$ is
dense in $K_{\alpha}$.

It can be verified that for an inner function $\alpha$ the formula
\begin{equation}\label{numerek3}
C_{\alpha}f(z)=\alpha(z)\overline{z}\overline{f(z)},\quad |z|=1,
\end{equation}
defines an antilinear, isometric involution on $L^2(\partial\mathbb{D})$ and that the so-called conjugation $C_{\alpha}$ preserves $K_{\alpha}$ (see \cite[Subection 2.3]{s}). It is easy to see that the conjugate kernel
$\widetilde{k}_{w}^{\alpha}=C_{\alpha}{k}_{w}^{\alpha}$ is given by
$$\widetilde{k}_{w}^{\alpha}(z)=\frac{\alpha(z)-\alpha(w)}{z-w}.$$

The unilateral shift is an example of a Toeplitz operator. The classical Toeplitz operator $T_{\varphi}$ on $H^2$ with symbol
$\varphi\in L^2(\partial\mathbb{D})$ is defined by
$$T_{\varphi}f=P(\varphi f).$$
The operator $T_{\varphi}$ is densely defined on bounded functions and can be extended to a bounded operator on $H^2$ if and only if $\varphi\in
L^{\infty}(\partial\mathbb{D})$. We have $S=T_z$.


In 2007 in his paper \cite{s} D. Sarason initiated the study of the class compressions of classical Toeplitz operators to the model spaces, that is, the truncated Toeplitz operators.

Let $\alpha$ be an inner function and let $P_{\alpha}$ denote the orthogonal projection from $L^2(\partial\mathbb{D})$ onto $K_{\alpha}$. A truncated Toeplitz operator on $K_{\alpha}$ with a symbol $\varphi\in
L^2(\partial\mathbb{D})$ is given by
$$A_{\varphi}^{\alpha}f=P_{\alpha}(\varphi f).$$
The operator $A_{\varphi}^{\alpha}$ is defined on a dense subset $K_{\alpha}^{\infty}$ of $K_{\alpha}$. Clearly, if $\varphi\in
L^{\infty}(\partial\mathbb{D})$, then $A_{\varphi}^{\alpha}$ extends to a bounded operator on $K_{\alpha}$. However, the boundednes of the symbol is not a necessary condition for the boundednes of the operator. Let $\mathscr{T}(\alpha)$ denote the set of all bounded truncated Toeplitz operators on $K_{\alpha}$.

Among many interesting results concerning truncated Toeplitz operators (see \cite{si,bros, gar,gar2,gar3}) there are the characterizations of these operators, given in \cite{s} and \cite{w}.

It is known that the classical Toeplitz operators are those bounded linear operators $T\colon H^2\rightarrow H^2$ that satisfy $T-S^{*}TS=0$. In \cite{s} Sarason gave a similar description of truncated Toeplitz operators on $K_{\alpha}$ using the compressed shift $S_{\alpha}=A_{z}^{\alpha}$. He proved that a bounded linear operator $A\colon K_{\alpha}\rightarrow K_{\alpha}$ belongs to $\mathscr{T}(\alpha)$ if and only if there exist $\chi,\psi\in K_{\alpha}$ such that
\begin{equation*}
A-S_{\alpha}^{*}AS_{\alpha}=\psi\otimes \widetilde{k}_{0}^{\alpha}+\widetilde{k}_{0}^{\alpha}\otimes \chi
\end{equation*}
(here $f\otimes g$ is the standard rank-one operator, $f\otimes g(h)=\langle h,g\rangle f$). In fact, \cite{s} contains a more general version of this result, where $S_{\alpha}$ is replaced by the so-called modified compressed shift. For $c\in\mathbb{C}$ the modified compressed shift $S_{\alpha,c}$ is given by
$$S_{\alpha,c}=S_{\alpha}+c(k_0^{\alpha}\otimes \widetilde{k}_0^{\alpha}).$$
An operator $A$ on $K_{\alpha}$ belongs to $\mathscr{T}(\alpha)$ if and only if for some $c\in\mathbb{C}$ there exist $\chi_c,\psi_c\in K_{\alpha}$ such that
\begin{equation*}
A-S_{\alpha,c}^{*}AS_{\alpha,c}=\psi_c\otimes \widetilde{k}_{0}^{\alpha}+\widetilde{k}_{0}^{\alpha}\otimes \chi_c.
\end{equation*}

The equation $T-S^{*}TS=0$ can be expressed in terms of the matrix representation of $T$ with respect to the monomial basis of $H^2$. It then states that $T$ is a Toeplitz operator if and only if the matrix representing $T$ is a Toeplitz matrix, i.e., an infinite matrix with constant diagonals. It follows that if $\alpha(z)=z^n$, then the operators from $\mathscr{T}(\alpha)$ are represented by finite Toeplitz matrices.

J. A. Cima, W. T. Ross and W.R. Wogen \cite{w} considered the matrix representations of truncated Toeplitz operators in the case when $K_{\alpha}$ is a finite-dimensional model space, that is, when $\alpha$ is a finite Blaschke product. In particular, they proved that if $\alpha$ is a finite Blaschke product with distinct zeros $a_1,\ldots,a_m$, then the matrix representing a truncated Toeplitz operator with respect to the reproducing kernel basis $\{k_{a_1}^{\alpha},\ldots,k_{a_m}^{\alpha}\}$ is completely determined by its entries along the main diagonal and the first row. They also gave similar characterizations in terms of matrix representaitons with respect to the so-called Clark bases and modified Clark bases (for detailed definitions see the following section).

Recently, in \cite{ptak} and \cite{part} the so-called asymmetric truncated
Toeplitz operators were introduced. Let $\alpha$, $\beta$ be two inner functions. An asymmetric
truncated Toeplitz operator $A_{\varphi}^{\alpha,\beta}$ on $K_{\alpha}$ with symbol $\varphi\in L^2(\partial\mathbb{D})$ is defined by
$$A_{\varphi}^{\alpha,\beta}f=P_{\beta}(\varphi f),\quad f\in K_{\alpha}.$$
Let
$$\mathscr{T}(\alpha,\beta)=\{A_{\varphi}^{\alpha,\beta}\ \colon\ \varphi\in
L^2(\partial\mathbb{D})\ \mathrm{and}\ A_{\varphi}^{\alpha,\beta}\
\mathrm{is\ bounded}\}.$$
Obviously, $A_{\varphi}^{\alpha,\alpha}=A_{\varphi}^{\alpha}$ and $\mathscr{T}(\alpha,\alpha)=\mathscr{T}(\alpha)$.

A characterization of operators form $\mathscr{T}(\alpha,\beta)$ in terms of compressed shifts was obtained in \cite{ptak} for $\alpha$ and $\beta$ such that $\beta$ divides $\alpha$ ($\alpha/\beta$ is also an inner function), whereas the matrix representations of asymmetric truncated Toeplitz operators on finite-dimensional spaces were discussed in \cite{blicharz2}.

In this paper we continue the study of asymmetric truncated Toeplitz operators acting between finite-dimensional model spaces. In Section 2 we cite some results from \cite{blicharz2} and then use these results in Section 3 to prove that the operators from $\mathscr{T}(\alpha,\beta)$ (in the finite-dimensional setting) can be characterized in terms of the modified compressed shifts $S_{\alpha,a}$, $S_{\beta,b}$. We also mention the characterization of asymmetric truncated Toeplitz operators in terms of the shift invariance.

In Section 4 we investigate the rank-one asymmetric truncated Toeplitz operators. We describe all the rank-one operators from $\mathscr{T}(\alpha,\beta)$ for finite Blaschke products $\alpha$ and $\beta$, and point out the cases when the given description differs form the one given by Sarason for rank-one truncated Toeplitz oprators.

For the reminder of the paper we assume that the spaces $K_{\alpha}$ and $K_{\beta}$ are finite-dimensional, that is, that $\alpha$ and $\beta$ are two finite Blaschke products of degree $m$ and $n$, respectively. Moreover, we assume that $m>0$ and $n>0$.

\section{Preliminaries}

For any $\lambda_1\in\partial\mathbb{D}$ define
\begin{equation}\label{1uj}
U_{\lambda_1}^{\alpha}=S_{\alpha}+\frac{\lambda_1+\alpha(0)}{1-|\alpha(0)|^2}(k_0^{\alpha}\otimes\widetilde{k}_0^{\alpha}).
\end{equation}
Then $U_{\lambda_1}^{\alpha}\colon K_{\alpha}\rightarrow K_{\alpha}$ is unitary and, in fact, all one-dimesional unitary perturbations of the compressed shift $S_{\alpha}$ are of this form. This result is due to D.N. Clark \cite{c} who also described the spectrum of $U_{\lambda_1}^{\alpha}$. In particular, he showed that the set of eigenvalues of $U_{\lambda_1}^{\alpha}$ consists of those $\eta\in\partial\mathbb{D}$ at which $\alpha$ has a finite angular derivative and
\begin{equation}\label{alf}
\alpha(\eta)=\alpha_{\lambda_1}=\frac{\lambda_1+\alpha(0)}{1+\overline{\alpha(0)}\lambda_1}.
\end{equation}
The corresponding eigenvectors are given by
\begin{equation*}
k_{\eta}^{\alpha}(z)=\frac{1-\overline{\alpha(\eta)}\alpha(z)}{1-\overline{\eta}z}.
\end{equation*}

Actually, $k_{\eta}^{\alpha}$ belongs to $K_{\alpha}$ for every $\eta\in\partial\mathbb{D}$ such that $\alpha$ has an angular derivative in the sense of Carath\'{e}odory at $\eta$ ($\alpha$ has an angular derivative at $\eta$ and $|\alpha(\eta)|=1$). In that case every $f\in K_{\alpha}$ has a non-tangential limit $f(\eta)$ at $\eta$ and $f(\eta)=\langle f,k_{\eta}^{\alpha}\rangle$ (see \cite[Thm. 7.4.1]{bros}). Note also that $\|k_{\eta}^{\alpha}\|=\sqrt{|\alpha'(\eta)|}$.

Here $\alpha$ is a finite Blaschke product, so $k_{\eta}^{\alpha}$ belongs to $K_{\alpha}$ for all $\eta\in\partial\mathbb{D}$. Since the degree of $\alpha$ is $m$, the equation \eqref{alf} has precisely $m$ distinct solutions $\eta_1,\ldots,\eta_m$ on the unit circle $\partial\mathbb{D}$ (see \cite[p. 6]{gar}). Moreover, since the dimension of $K_{\alpha}$ is $m$, the orthogonal set of corresponding eigenvectors $\{k_{\eta_1}^{\alpha},\ldots,k_{\eta_m}^{\alpha}\}$ forms a basis for $K_{\alpha}$. The set of normalized kernel functions $\{v_{\eta_1}^{\alpha},\ldots,v_{\eta_m}^{\alpha}\}$,
\begin{equation*}
v_{\eta_j}^{\alpha}=\|k_{\eta_j}^{\alpha}\|^{-1}k_{\eta_j}^{\alpha},\quad j=1,\ldots,m,
\end{equation*}
 is called the Clark basis corresponding to $\lambda_1$ (see \cite{c} for more details).

 Observe that for every $f\in K_{\alpha}$ we have
 \begin{displaymath}
 f=\sum_{j=1}^m\langle f,v_{\eta_j}^{\alpha}\rangle v_{\eta_j}^{\alpha}=\sum_{j=1}^m\frac{f(\eta_j)}{\sqrt{|\alpha'(\eta_j)|}} v_{\eta_j}^{\alpha}=\sum_{j=1}^m\frac{f(\eta_j)}{|\alpha'(\eta_j)|} k_{\eta_j}^{\alpha}.
 \end{displaymath}

If we put
\begin{equation*}
e_{\eta_j}^{\alpha}=\omega_{j}^{\alpha}v_{\eta_j}^{\alpha},\quad j=1,\ldots,m,
\end{equation*}
where
\begin{equation*}
\omega_{j}^{\alpha}=e^{-\frac{i}{2}(\mathrm{arg}\eta_j-\mathrm{arg}\lambda_1)},\quad j=1,\ldots,m,
\end{equation*}
then the basis $\{e_{\eta_1}^{\alpha},\ldots,e_{\eta_m}^{\alpha}\}$ has an additional property that $$C_{\alpha}e_{\eta_j}^{\alpha}=e_{\eta_j}^{\alpha},\quad j=1,\ldots,m,$$
where $C_{\alpha}$ is the conjugation given by \eqref{numerek3}. The basis $\{e_{\eta_1}^{\alpha},\ldots,e_{\eta_m}^{\alpha}\}$ is called the modified Clark basis.

Similarly, for $\lambda_2\in\partial\mathbb{D}$ there are $n$ orthogonal eigenvectors $k_{\zeta_1}^{\beta},\ldots,k_{\zeta_n}^{\beta}$ of the unitary operator
\begin{equation}\label{2uj}
U_{\lambda_2}^{\beta}=S_{\beta}+\frac{\lambda_2+\beta(0)}{1-|\beta(0)|^2}(k_0^{\beta}\otimes\widetilde{k}_0^{\beta}),
\end{equation}
each corresponding to a different solution $\zeta_j$, $j=1,\ldots,n$, of the equation
\begin{equation}\label{bet}
\beta(\zeta)=\beta_{\lambda_2}=\frac{\lambda_2+\beta(0)}{1+\overline{\beta(0)}\lambda_2}.
\end{equation}
The Clark basis $\{v_{\zeta_1}^{\beta},\ldots,v_{\zeta_n}^{\beta}\}$ and modified Clark basis $\{e_{\zeta_1}^{\beta},\ldots,e_{\zeta_n}^{\beta}\}$ for $K_{\beta}$ are defined analogously by
\begin{equation*}
v_{\zeta_j}^{\beta}=\|k_{\zeta_j}^{\beta}\|^{-1}k_{\zeta_j}^{\beta},\quad j=1,\ldots,n,
\end{equation*}
and
\begin{equation*}
e_{\zeta_j}^{\beta}=\omega_{j}^{\beta}v_{\zeta_j}^{\beta},\quad j=1,\ldots,n,
\end{equation*}
where
\begin{equation*}
\omega_{j}^{\beta}=e^{-\frac{i}{2}(\mathrm{arg}\zeta_j-\mathrm{arg}\lambda_2)},\quad j=1,\ldots,n.
\end{equation*}

Clearly, the equations \eqref{alf} and \eqref{bet} may have some solutions in common. In that case we assume that these solutions are precisely the numbers $\eta_j=\zeta_j$ for $j\leq l$ (with $l=0$ if there are no solutions in common).

In what follows we will use the characterization of operators form $\mathscr{T}(\alpha,\beta)$ in terms of their matrix representations with respect to the Clark bases.

\begin{thm}[\cite{blicharz2}, Thm. 3.4]\label{macierz3}
Let $\alpha$ and $\beta$ be two finite Blaschke products of degree $m>0$ and $n>0$, respectively. Let $\{v_{\eta_1}^{\alpha},\ldots, v_{\eta_m}^{\alpha} \}$ be the Clark basis for $K_{\alpha}$ corresponding to $\lambda_1\in\partial\mathbb{D}$, let $\{v_{\zeta_1}^{\beta},\ldots, v_{\zeta_n}^{\beta} \}$ be the Clark basis for $K_{\beta}$ corresponding to $\lambda_2\in\partial\mathbb{D}$ and assume that the sets $\{\eta_1,\ldots,\eta_m\}$, $\{\zeta_1,\ldots,\zeta_n\}$ have precisely $l$ elements in common: $\eta_j=\zeta_j$ for $j\leq l$ ($l=0$ if there are no elements in common). Finally, let $A$ be any linear transformation from $K_{\alpha}$ into $K_{\beta}$. If $M_A=(r_{s,p})$ is the matrix representation of $A$ with respect to the bases $\{v_{\eta_1}^{\alpha},\ldots, v_{\eta_m}^{\alpha} \}$ and $\{v_{\zeta_1}^{\beta},\ldots, v_{\zeta_n}^{\beta} \}$, and
  \vspace{0.2cm}
\begin{itemize}
\item[(a)]$l=0 $, then $A\in \mathscr{T}(\alpha,\beta)$ if and only if
\begin{equation}\label{c1}
\begin{split}
r_{s,p}=&\left(\frac{\sqrt{|\alpha'(\eta_1)|}}{\sqrt{|\alpha'(\eta_p)|}}\frac{\eta_p}{\eta_1}\frac{\eta_1-\zeta_s}{\eta_p-\zeta_s}r_{s,1}+\frac{\sqrt{|\alpha'(\eta_1)|}\sqrt{|\beta'(\zeta_1)|}}{\sqrt{|\alpha'(\eta_p)|}\sqrt{|\beta'(\zeta_s)|}}\frac{\eta_p}{\eta_1}\frac{\zeta_1-\eta_1}{\eta_p-\zeta_s}r_{1,1}\right.
\\&\phantom{=\frac1{\sqrt{|\alpha'(\eta_p)|}}}+\left.\frac{\sqrt{|\beta'(\zeta_1)|}}{\sqrt{|\beta'(\zeta_s)|}}\frac{\eta_p-\zeta_1}{\eta_p-\zeta_s}r_{1,p}\right)
\end{split}
\end{equation} for all $1\leq p \leq m$ and $1\leq  s\leq n$;
\vspace{0.5cm}
\item[(b)]$l>0$, then $A\in \mathscr{T}(\alpha,\beta)$ if and only if
\begin{equation}\label{c2}
r_{s,p}=\left(\frac{\sqrt{|\alpha'(\eta_s)|}\sqrt{|\beta'(\zeta_1)|}}{\sqrt{|\alpha'(\eta_p)|}\sqrt{|\beta'(\zeta_s)|}}\frac{\eta_p}{\eta_s}\frac{\eta_1-\zeta_s}{\eta_p-\zeta_s}r_{1,s}+\frac{\sqrt{|\beta'(\zeta_1)|}}{\sqrt{|\beta'(\zeta_s)|}}\frac{\eta_p-\zeta_1}{\eta_p-\zeta_s}r_{1,p}\right)
\end{equation} for all $p,s$ such that $1\leq p \leq m$, $1\leq  s\leq l$, $s\neq p$, and
\begin{equation}\label{c3}
r_{s,p}=\left(\frac{\sqrt{|\alpha'(\eta_1)|}}{\sqrt{|\alpha'(\eta_p)|}}\frac{\eta_p}{\eta_1}\frac{\eta_1-\zeta_s}{\eta_p-\zeta_s}r_{s,1}+\frac{\sqrt{|\beta'(\zeta_1)|}}{\sqrt{|\beta'(\zeta_s)|}}\frac{\eta_p-\zeta_1}{\eta_p-\zeta_s}r_{1,p}\right)
\end{equation} for all $p,s$ such that $1\leq p \leq m$, $l<  s\leq n$.
\end{itemize}
\end{thm}

\section{Characterizations using rank-two operators}

In this section we prove that if $\alpha$ and $\beta$ are finite Blaschke products, then a characterization in terms of modified compressed shifts, similar to that form \cite[Thm. 10.1]{s}, can be given for operators from $\mathscr{T}(\alpha,\beta)$.

\begin{thm}\label{crit}
Let $\alpha$, $\beta$ be two finite Blaschke products and let $a$, $b$ be two complex numbers. A linear operator $A$ from $K_{\alpha}$ into $K_{\beta}$ belongs to $\mathscr{T}(\alpha,\beta)$ if and only if there are functions $\chi_{a,b}\in K_{\alpha}$ and $\psi_{a,b}\in K_{\beta}$ such that
\begin{equation}\label{con1}
A-S_{\beta,b}A S_{\alpha,a}^{*}=\psi_{a,b}\otimes k_{0}^{\alpha}+k_{0}^{\beta}\otimes\chi_{a,b}.
\end{equation}
\end{thm}
\begin{proof}
Let $A$ be a linear operator from $K_{\alpha}$ into $K_{\beta}$ and let $a$, $b$ be two complex numbers.

Note first that if $a_1$, $b_1$ are any other two complex numbers, then the operator $A$ satisfies \eqref{con1} with some $\chi_{a,b}\in K_{\alpha}$, $\psi_{a,b}\in K_{\beta}$ if and only if there exist $\chi_{a_1,b_1}\in K_{\alpha}$, $\psi_{a_1,b_1}\in K_{\beta}$ such that
\begin{equation}\label{con11}
A-S_{\beta,b_1}A S_{\alpha,a_1}^{*}=\psi_{a_1,b_1}\otimes k_{0}^{\alpha}+k_{0}^{\beta}\otimes\chi_{a_1,b_1}.
\end{equation}
Indeed, using
\begin{displaymath}
\begin{split}
S_{\beta,b}A S_{\alpha,a}^{*}&=[S_{\beta}+b(k_{0}^{\beta}\otimes \widetilde{k}_{0}^{\beta})]A[S_{\alpha}^{*}+\overline{a}(\widetilde{k}_{0}^{\alpha}\otimes {k}_{0}^{\alpha})]\\
&=S_{\beta}A S_{\alpha}^{*}+b (k_{0}^{\beta}\otimes S_{\alpha}A^{*}\widetilde{k}_{0}^{\beta})+\overline{a} (S_{\beta}A\widetilde{k}_{0}^{\alpha}\otimes k_{0}^{\alpha})+b\overline{a}\langle A\widetilde{k}_{0}^{\alpha},\widetilde{k}_{0}^{\beta}\rangle (k_{0}^{\beta}\otimes {k}_{0}^{\alpha}),
\end{split}
\end{displaymath}
and an analogous formula for $S_{\beta,b_1}A S_{\alpha,a_1}^{*}$, it is easy to verify that there exist functions $\chi\in K_{\alpha}$ and $\psi\in K_{\beta}$ such that
\begin{equation*}
A-S_{\beta,b_1}A S_{\alpha,a_1}^{*}=A-S_{\beta,b}A S_{\alpha,a}^{*}+ \psi\otimes k_{0}^{\alpha}+k_{0}^{\beta}\otimes\chi.
\end{equation*}
This implies the equivalence of \eqref{con1} and \eqref{con11}.

In other words, instead of the condition \eqref{con1} we can consider the condition \eqref{con11} with a suitable $a_1$, $b_1$.


Fix $\lambda_1,\lambda_2\in\partial\mathbb{D}$ and let
$$a_1=\frac{\alpha(0)+\lambda_1}{1-|\alpha(0)|^2}\quad\mathrm{and}\quad b_1=\frac{\beta(0)+\lambda_2}{1-|\beta(0)|^2}.$$
Since
$$S_{\alpha,a_1}=U_{\lambda_1}^{\alpha}\quad\mathrm{and}\quad S_{\beta,b_1}=U_{\lambda_2}^{\beta},$$
where $U_{\lambda_1}^{\alpha}$ and $U_{\lambda_2}^{\beta}$ are the Clark operators defined by \eqref{1uj} and \eqref{2uj}, the condition \eqref{con11} becomes
\begin{equation}\label{con2}
A-U_{\lambda_2}^{\beta}A (U_{\lambda_1}^{\alpha})^{*}=\psi\otimes k_{0}^{\alpha}+k_{0}^{\beta}\otimes\chi,
\end{equation}
with $\chi=\chi_{a_1,b_1}\in K_{\alpha}$ and $\psi=\psi_{a_1,b_1}\in K_{\beta}$.

Let $\{v_{\eta_1}^{\alpha},\ldots,v_{\eta_m}^{\alpha}\}$ and $\{v_{\zeta_1}^{\beta},\ldots,v_{\zeta_n}^{\beta}\}$ be the Clark bases of $K_{\alpha}$ and $K_{\beta}$ corresponding to $\lambda_1$ and $\lambda_2$, respectively. Moreover, let $M_A=(r_{s,p})$ be the matrix representation of $A$ with respect to these bases.

We now show that $A\in\mathscr{T}(\alpha,\beta)$ if and only if it satisfies \eqref{con2}. To this end we, express the condition \eqref{con2} in terms of the matrix representation $M_A$ and use Theorem \ref{macierz3}.

The operators in \eqref{con2} are equal if and only if their matrix representations with respect to $\{v_{\eta_1}^{\alpha},\ldots,v_{\eta_m}^{\alpha}\}$ and $\{v_{\zeta_1}^{\beta},\ldots,v_{\zeta_n}^{\beta}\}$ are equal. The latter holds if and only if
\begin{equation}\label{x1}
	\begin{array}{l}
	\langle Av_{\eta_p}^{\alpha},v_{\zeta_s}^{\beta}\rangle-\langle A (U_{\lambda_1}^{\alpha})^{*}v_{\eta_p}^{\alpha},(U_{\lambda_2}^{\beta})^{*}v_{\zeta_s}^{\beta}\rangle\\
	\phantom{\langle Av_{\eta_p}^{\alpha},\rangle}=\langle v_{\eta_p}^{\alpha}, k_{0}^{\alpha}\rangle\langle\psi,v_{\zeta_s}^{\beta}\rangle+\langle v_{\eta_p}^{\alpha},\chi\rangle\langle k_{0}^{\beta},v_{\zeta_s}^{\beta}\rangle
	\end{array}\quad\mathrm{for\ all}\quad \begin{array}{l}
1\leq p\leq m\\
\phantom{1} 1\leq s\leq n\end{array}.
\end{equation}
Recall that
$$r_{s,p}=\langle Av_{\eta_p}^{\alpha},v_{\zeta_s}^{\beta}\rangle.$$
This, together with
$$(U_{\lambda_1}^{\alpha})^{*}v_{\eta_p}^{\alpha}=\overline{\eta}_pv_{\eta_p}^{\alpha},\quad (U_{\lambda_2}^{\beta})^{*}v_{\zeta_s}^{\beta}=\overline{\zeta}_sv_{\zeta_s}^{\beta},$$
and the reproducing property of the eigenvectors of $U_{\lambda_1}^{\alpha}$ and $U_{\lambda_2}^{\beta}$, implies that the system of equations \eqref{x1} is equivalent to the following
\begin{equation}\label{x2}
\begin{array}{l}
(1-\overline{\eta}_p\zeta_s)r_{s,p}\sqrt{|\alpha'(\eta_p)|}\sqrt{|\beta'(\zeta_s)|}\\
\phantom{(1-\overline{\eta}_p\zeta_s)r_{s,p}}=\psi(\zeta_s)\overline{k_{0}^{\alpha}(\eta_p)}+k_{0}^{\beta}(\zeta_s)\overline{\chi(\eta_p)}\end{array},\quad
\begin{array}{l}
1\leq p\leq m\\
\phantom{1} 1\leq s\leq n
\end{array}.
\end{equation}


Assume that $A$ satisfies \eqref{con2} with some $\chi\in K_{\alpha}$ and $\psi\in K_{\beta}$. By the above, the matrix representation $M_A=(r_{s,p})$ satisfies \eqref{x2}. If $l=0$, then
\begin{displaymath}
	\begin{split}
	 \left(\frac{\sqrt{|\alpha'(\eta_1)|}}{\sqrt{|\alpha'(\eta_p)|}}\frac{\eta_p}{\eta_1}\right.&\left.\frac{\eta_1-\zeta_s}{\eta_p-\zeta_s}r_{s,1}+\frac{\sqrt{|\alpha'(\eta_1)|}\sqrt{|\beta'(\zeta_1)|}}{\sqrt{|\alpha'(\eta_p)|}\sqrt{|\beta'(\zeta_s)|}}\frac{\eta_p}{\eta_1}\frac{\zeta_1-\eta_1}{\eta_p-\zeta_s}r_{1,1}\right.\\
	 &\left.+\frac{\sqrt{|\beta'(\zeta_1)|}}{\sqrt{|\beta'(\zeta_s)|}}\frac{\eta_p-\zeta_1}{\eta_p-\zeta_s}r_{1,p}\right)=\frac{\psi(\zeta_s)\overline{k_{0}^{\alpha}(\eta_1)}+k_{0}^{\beta}(\zeta_1)\overline{\chi(\eta_p)}}{\sqrt{|\alpha'(\eta_p)|}\sqrt{|\beta'(\zeta_s)|}(1-\overline{\eta}_p\zeta_s)}\\
	 &=\frac{\psi(\zeta_s)\overline{k_{0}^{\alpha}(\eta_p)}+k_{0}^{\beta}(\zeta_s)\overline{\chi(\eta_p)}}{\sqrt{|\alpha'(\eta_p)|}\sqrt{|\beta'(\zeta_s)|}(1-\overline{\eta}_p\zeta_s)}	 =r_{s,p},
	\end{split}
\end{displaymath}
for all $1\leq p\leq m$ and $1\leq s\leq n$. This follows from \eqref{x2} and the fact that
$$k_{0}^{\alpha}(\eta_p)=1-\overline{\alpha(0)}\alpha_{\lambda_1}=k_{0}^{\alpha}(\eta_1)\quad\mathrm{and}\quad k_{0}^{\beta}(\zeta_s)=1-\overline{\beta(0)}\alpha_{\lambda_2}=k_{0}^{\beta}(\zeta_1).$$
By Theorem \ref{macierz3}, $A$ belongs to $\mathscr{T}(\alpha,\beta)$.

If $l>0$, then \eqref{x2} implies that
$$\psi(\zeta_s)\overline{k_{0}^{\alpha}(\eta_s)}+k_{0}^{\beta}(\zeta_s)\overline{\chi(\eta_s)}=0$$
for all $1\leq s\leq l$. It follows that
\begin{equation*}
\begin{split}
\left(\frac{\sqrt{|\alpha'(\eta_s)|}}{\sqrt{|\alpha'(\eta_p)|}}\right.&\left.\frac{\sqrt{|\beta'(\zeta_1)|}}{\sqrt{|\beta'(\zeta_s)|}}\frac{\eta_p}{\eta_s}\frac{\eta_1-\zeta_s}{\eta_p-\zeta_s}r_{1,s}+\frac{\sqrt{|\beta'(\zeta_1)|}}{\sqrt{|\beta'(\zeta_s)|}}\frac{\eta_p-\zeta_1}{\eta_p-\zeta_s}r_{1,p}\right)\\
&=\frac{\psi(\zeta_1)\overline{k_{0}^{\alpha}(\eta_p)}+k_{0}^{\beta}(\zeta_1)\overline{\chi(\eta_p)}-\psi(\zeta_1)\overline{k_{0}^{\alpha}(\eta_s)}-k_{0}^{\beta}(\zeta_1)\overline{\chi(\eta_s)}}{\sqrt{|\alpha'(\eta_p)|}\sqrt{|\beta'(\zeta_s)|}(1-\overline{\eta}_p\zeta_s)}=r_{s,p}
\end{split}
\end{equation*} for all $1\leq p \leq m$, $1\leq  s\leq l$, $s\neq p$. Similarly,
\begin{equation*}
\begin{split}
\left(\frac{\sqrt{|\alpha'(\eta_1)|}}{\sqrt{|\alpha'(\eta_p)|}}\frac{\eta_p}{\eta_1}\right.&\left.\frac{\eta_1-\zeta_s}{\eta_p-\zeta_s}r_{s,1}+\frac{\sqrt{|\beta'(\zeta_1)|}}{\sqrt{|\beta'(\zeta_s)|}}\frac{\eta_p-\zeta_1}{\eta_p-\zeta_s}r_{1,p}\right)\\
&=\frac{\psi(\zeta_s)\overline{k_{0}^{\alpha}(\eta_1)}+k_{0}^{\beta}(\zeta_s)\overline{\chi(\eta_1)}+\psi(\zeta_1)\overline{k_{0}^{\alpha}(\eta_p)}+k_{0}^{\beta}(\zeta_1)\overline{\chi(\eta_p)}}{\sqrt{|\alpha'(\eta_p)|}\sqrt{|\beta'(\zeta_s)|}(1-\overline{\eta}_p\zeta_s)}=r_{s,p}.
\end{split}
\end{equation*}
for all $1\leq p \leq m$ and $l<  s\leq n$. By Theorem \ref{macierz3} again, $A$ belongs to $\mathscr{T}(\alpha,\beta)$.

Assume now that $A$ belongs to $\mathscr{T}(\alpha,\beta)$. By the above, to show that there exist $\chi\in K_{\alpha}$ and $\psi\in K_{\beta}$ such that \eqref{con2} holds it is enough to find finite sequances $\{\chi_1,\ldots,\chi_m\}$ and $\{\psi_1,\ldots,\psi_n\}$ such that
\begin{equation}\label{x3}
\psi_s\overline{k_{0}^{\alpha}(\eta_p)}+k_{0}^{\beta}(\zeta_s)\overline{\chi_p}=(1-\overline{\eta}_p\zeta_s)r_{s,p}\sqrt{|\alpha'(\eta_p)|}\sqrt{|\beta'(\zeta_s)|},\quad
\begin{array}{l}
1\leq p\leq m\\
\phantom{1} 1\leq s\leq n
\end{array}.
\end{equation}
Indeed, if $\{\chi_1,\ldots,\chi_m\}$ and $\{\psi_1,\ldots,\psi_n\}$ satisfy \eqref{x3}, then $A$ satisfies \eqref{con2} with $\chi$ and $\psi$ given by
\begin{equation}\label{cp}
\chi=\sum_{p=1}^m\frac{\chi_p}{\sqrt{|\alpha'(\eta_p)|}}v_{\eta_p}^{\alpha},\quad \psi=\sum_{s=1}^n\frac{\psi_s}{\sqrt{|\beta'(\zeta_s)|}}v_{\zeta_s}^{\beta}.
\end{equation}
We now find such $\{\chi_1,\ldots,\chi_m\}$ and $\{\psi_1,\ldots,\psi_n\}$.

If $l=0$, then, by Theorem \ref{macierz3}(a), we only need to consider those equations from \eqref{x3}, which correspond to the elements $r_{1,p}$ and $r_{s,1}$ of the matrix representation of $A$. That is, $\{\chi_1,\ldots,\chi_m\}$ and $\{\psi_1,\ldots,\psi_n\}$ satisfy \eqref{x3} if and only if they satisfy the following system of equations
\begin{equation}\label{x4}
\begin{cases}
\psi_1\overline{k_{0}^{\alpha}(\eta_p)}+k_{0}^{\beta}(\zeta_1)\overline{\chi_p}=(1-\overline{\eta}_p\zeta_1)r_{1,p}\sqrt{|\alpha'(\eta_p)|}\sqrt{|\beta'(\zeta_1)|}&\mathrm{for\ all}\ 1\leq p\leq m\\
\psi_s\overline{k_{0}^{\alpha}(\eta_1)}+k_{0}^{\beta}(\zeta_s)\overline{\chi_1}=(1-\overline{\eta}_1\zeta_s)r_{s,1}\sqrt{|\alpha'(\eta_1)|}\sqrt{|\beta'(\zeta_s)|}&\mathrm{for\ all}\ 1< s\leq n
\end{cases}.
\end{equation}
This can easily be verified using \eqref{c1}.

Fix an arbitrary $\psi_1$. Then the desired sequences are
\begin{equation*}
\begin{cases}
\chi_p=\frac{(1-\eta_p\overline{\zeta}_1)\overline{r}_{1,p}\sqrt{|\alpha'(\eta_p)|}\sqrt{|\beta'(\zeta_1)|}-\overline{\psi}_1k_{0}^{\alpha}(\eta_p)}{\overline{k_{0}^{\beta}(\zeta_1)}}&\mathrm{for}\quad 1\leq p\leq m\\
\psi_s=\frac{(1-\overline{\eta}_1\zeta_s)r_{s,1}\sqrt{|\alpha'(\eta_1)|}\sqrt{|\beta'(\zeta_s)|}-k_{0}^{\beta}(\zeta_s)\overline{\chi_1}}{\overline{k_{0}^{\alpha}(\eta_1)}}&\mathrm{for}\quad 1< s\leq n
\end{cases},
\end{equation*}
and the functions $\chi$ and $\psi$ in \eqref{con2} are given by \eqref{cp}.

Similarly, if $l>0$, then, by Theorem \ref{macierz3}(b), $\{\chi_1,\ldots,\chi_m\}$ and $\{\psi_1,\ldots,\psi_n\}$ satisfy \eqref{x3} if and only if they satisfy
\begin{equation}\label{x5}
\begin{cases}
\psi_1\overline{k_{0}^{\alpha}(\eta_p)}+k_{0}^{\beta}(\zeta_1)\overline{\chi_p}=(1-\overline{\eta}_p\zeta_1)r_{1,p}\sqrt{|\alpha'(\eta_p)|}\sqrt{|\beta'(\zeta_1)|}&\mathrm{for\ all}\quad 1\leq p\leq m\\
\psi_s\overline{k_{0}^{\alpha}(\eta_s)}+k_{0}^{\beta}(\zeta_s)\overline{\chi_s}=0&\mathrm{for\ all}\quad 1\leq s\leq l\\
\psi_s\overline{k_{0}^{\alpha}(\eta_1)}+k_{0}^{\beta}(\zeta_s)\overline{\chi_1}=(1-\overline{\eta}_1\zeta_s)r_{s,1}\sqrt{|\alpha'(\eta_1)|}\sqrt{|\beta'(\zeta_s)|}&\mathrm{for\ all}\quad l< s\leq n
\end{cases}
\end{equation}
(consider those equations from \eqref{x3} with $r_{1,p}$ for $1\leq p\leq m$, $r_{s,s}$ for $1\leq s\leq l$ and $r_{s,1}$ for $l<s\leq n$).
Fix an arbitrary $\psi_1$. Then
\begin{equation*}
\begin{cases}
\chi_p=\frac{(1-\eta_p\overline{\zeta}_1)\overline{r}_{1,p}\sqrt{|\alpha'(\eta_p)|}\sqrt{|\beta'(\zeta_1)|}-\overline{\psi}_1k_{0}^{\alpha}(\eta_p)}{\overline{k_{0}^{\beta}(\zeta_1)}}&\mathrm{for}\quad 1\leq p\leq m\\
\psi_s=-\frac{k_{0}^{\beta}(\zeta_s)\overline{\chi_s}}{\overline{k_{0}^{\alpha}(\eta_s)}}&\mathrm{for}\quad 1\leq s\leq l\\
\psi_s=\frac{(1-\overline{\eta}_1\zeta_s)r_{s,1}\sqrt{|\alpha'(\eta_1)|}\sqrt{|\beta'(\zeta_s)|}-k_{0}^{\beta}(\zeta_s)\overline{\chi_1}}{\overline{k_{0}^{\alpha}(\eta_1)}}&\mathrm{for}\quad l< s\leq n
\end{cases}.
\end{equation*}
Again, $A$ satisfies \eqref{con2} with $\chi$ and $\psi$ given by \eqref{cp}. This completes the proof.
\end{proof}

\begin{cor}\label{crit2}
Let $\alpha$, $\beta$ be two finite Blaschke products and let $a$, $b$ be two complex numbers. A linear operator $A$ from $K_{\alpha}$ into $K_{\beta}$ belongs to $\mathscr{T}(\alpha,\beta)$ if and only if there are functions $\chi_{a,b}\in K_{\alpha}$ and $\psi_{a,b}\in K_{\beta}$ such that
\begin{equation}\label{con3}
A-S_{\beta,b}^{*}A S_{\alpha,a}=\psi_{a,b}\otimes \widetilde{k}_{0}^{\alpha}+\widetilde{k}_{0}^{\beta}\otimes\chi_{a,b}.
\end{equation}
\end{cor}
\begin{proof}
We first prove that the linear operator $A$ belongs to $\mathscr{T}(\alpha,\beta)$ if and only if $B=C_{\beta}AC_{\alpha}$ belongs to $\mathscr{T}(\alpha,\beta)$.

If $A$ belongs to $\mathscr{T}(\alpha,\beta)$, then $A=A_{\varphi}^{\alpha,\beta}$ for some $\varphi\in L^2(\partial\mathbb{D})$, and for every $f\in K_{\alpha}^{\infty}$ and $g\in K_{\beta}^{\infty}$,
\begin{displaymath}
\begin{split}
\langle Bf,g\rangle &=\langle C_{\beta}AC_{\alpha}f,g\rangle=\langle C_{\beta}g,AC_{\alpha}f\rangle\\
&=\langle \beta \overline{z g}, \varphi\alpha \overline{z f}\rangle=\langle \overline{\alpha\varphi}\beta f,g\rangle=\langle A_{\overline{\alpha\varphi}\beta}^{\alpha,\beta}f,g\rangle.
\end{split}
\end{displaymath}
Hence
$$B=C_{\beta}AC_{\alpha}=A_{\overline{\alpha\varphi}\beta}^{\alpha,\beta}.$$
On the other hand, if $B=C_{\beta}AC_{\alpha}$ belongs to $\mathscr{T}(\alpha,\beta)$, then, by the above, $C_{\beta}BC_{\alpha}$ also belongs to $\mathscr{T}(\alpha,\beta)$. However,
$$C_{\beta}BC_{\alpha}=C_{\beta}^2AC_{\alpha}^2=A,$$
and $A\in\mathscr{T}(\alpha,\beta)$.

We now prove that $A\in\mathscr{T}(\alpha,\beta)$ if and only if \eqref{con3} holds.

Assume that $A\in\mathscr{T}(\alpha,\beta)$. Then $C_{\beta}AC_{\alpha}\in\mathscr{T}(\alpha,\beta)$ and, by Theorem \ref{crit}, there are $\chi\in K_{\alpha}$ and $\psi\in K_{\beta}$ such that
\begin{equation}\label{con4}
C_{\beta}AC_{\alpha}-S_{\beta,b}C_{\beta}AC_{\alpha} S_{\alpha,a}^{*}=\psi\otimes k_{0}^{\alpha}+k_{0}^{\beta}\otimes\chi.
\end{equation}
Since
$$C_{\beta}S_{\beta,b}C_{\beta}=S_{\beta,b}^{*}\quad\mathrm{and}\quad C_{\alpha}S_{\alpha,a}^{*}C_{\alpha}=S_{\alpha,a},$$
we get
\begin{displaymath}
\begin{split}
A-S_{\beta,b}^{*}A S_{\alpha,a}&=C_{\beta}^2AC_{\alpha}^2-C_{\beta}S_{\beta,b}C_{\beta}A C_{\alpha}S_{\alpha,a}^{*}C_{\alpha}\\
&=C_{\beta}(\psi\otimes k_{0}^{\alpha}+k_{0}^{\beta}\otimes\chi)C_{\alpha}=\widetilde{\psi}\otimes \widetilde{k}_{0}^{\alpha}+\widetilde{k}_{0}^{\beta}\otimes\widetilde{\chi}.
\end{split}
\end{displaymath}
Thus $A$ satisfies \eqref{con3} with
$$\chi_{a,b}=\widetilde{\chi}\quad\mathrm{and}\quad \psi_{a,b}=\widetilde{\psi}.$$

A similar reasoning shows that if $A$ satisfies \eqref{con3}, then $C_{\beta}AC_{\alpha}$ satisfies \eqref{con1}. By Theorem \ref{crit}, $C_{\beta}AC_{\alpha}\in\mathscr{T}(\alpha,\beta)$ and so $A\in\mathscr{T}(\alpha,\beta)$. This completes the proof.

\end{proof}

\begin{cor}
If $\alpha$, $\beta$ are two finite Blaschke products, then $\mathscr{T}(\alpha,\beta)$ is closed in the weak operator topology.
\end{cor}
\begin{proof}
The proof is analogous to that of \cite[Thm. 4.2]{s} and is thus left for the reader.
\end{proof}

In \cite[p. 512]{s} D. Sarason considered the notion of shift invariance of operators on $K_{\alpha}$. In \cite{ptak2} this notion was generalized to operators form $K_{\alpha}$ into $K_{\beta}$. A bounded linear operator from $K_{\alpha}$ into $K_{\beta}$ is said to be shift invariant if
\begin{equation}\label{szin}
\langle ASf,Sg\rangle=\langle Af,g\rangle
\end{equation}
for all $f\in K_{\alpha}$ and $g\in K_{\beta}$ such that $Sf\in K_{\alpha}$ and $Sg\in K_{\beta}$.

The authors in \cite{ptak2} proved that if $\alpha$ and $\beta$ are two inner functions such that $\beta$ divides $\alpha$, then the set of all shift invariant operators form $K_{\alpha}$ into $K_{\beta}$ is equal to $\mathscr{T}(\alpha,\beta)$. For finite Blaschke products we have the following.

\begin{thm}
Let $\alpha$, $\beta$ be two finite Blaschke products and let $A$ be a linear operator from $K_{\alpha}$ into $K_{\beta}$. Then $A$ belongs to $\mathscr{T}(\alpha,\beta)$ if and only if $A$ is shift invariant.
\end{thm}
\begin{proof}
See \cite{ptak2}.
\end{proof}

\section{Rank-one asymmetric truncated Toeplitz operators}

In Section 5 of \cite{s} D. Sarason characterized all the truncated Toeplitz operators of rank-one. He proved (see \cite[Thm. 5.1]{s}) that every rank-one operator form $\mathscr{T}(\alpha)$ is a scalar multiple of $\widetilde{k}_w^{\alpha}\otimes {k}_w^{\alpha}$ or $k_w^{\alpha}\otimes \widetilde{k}_w^{\alpha}$ for some $w\in \overline{\mathbb{D}}$. Can a similar characterization be given for asymmetric truncated Toeplitz operators of rank-one?

It is known (see for instance \cite[Prop. 3.1]{blicharz1}) that the operators $\widetilde{k}_w^{\beta}\otimes {k}_w^{\alpha}$ and $k_w^{\beta}\otimes \widetilde{k}_w^{\alpha}$ belong to $\mathscr{T}(\alpha,\beta)$ for all $w\in \mathbb{D}$ and for such $w\in \partial\mathbb{D}$ that $\alpha$ and $\beta$ have an angular derivative in the sense of Carath\'{e}odory at $w$. So a natural question is: is every rank-one operator from $\mathscr{T}(\alpha,\beta)$ a scalar multiple of one of these? A simple example shows that this is not always the case.

\begin{exmp}
Fix $a\in\mathbb{D}$, $a\neq 0$, and let
$$\alpha(z)=-z\frac{a-z}{1-\overline{a}z}\frac{a+z}{1+\overline{a}z},\quad \beta(z)=z.$$
Then $\beta$ divides $\alpha$, $K_{\beta}$ is a one-dimensional linear space spanned by $k_0^{\beta}=1$ and $K_{\alpha}$ is a three-dimensional linear space spanned by $k_0^{\alpha}=1$, $k_a^{\alpha}=k_a$ and $k_{-a}^{\alpha}=k_{-a}$. Moreover, since $\alpha$ and $\beta$ are analytic in a domain containing $\overline{\mathbb{D}}$, the operators $\widetilde{k}_w^{\beta}\otimes {k}_w^{\alpha}$ and $k_w^{\beta}\otimes \widetilde{k}_w^{\alpha}$ belong to $\mathscr{T}(\alpha,\beta)$ for all $w\in \overline{\mathbb{D}}$. 

It is easy to verify that $A=1\otimes(1+k_a)$ is an element of $\mathscr{T}(\alpha,\beta)$. In fact,
$$A=1\otimes(1+k_a)=A_{1+\overline{k}_a}^{\alpha,\beta}.$$
We show that there is no $w\in\overline{\mathbb{D}}$ such that $A$ is a scalar multiple of $\widetilde{k}_w^{\beta}\otimes {k}_w^{\alpha}$ or
$k_w^{\beta}\otimes \widetilde{k}_w^{\alpha}$.

Assume first that
\begin{equation}\label{row1}
1\otimes(1+k_a)=c(\widetilde{k}_w^{\beta}\otimes {k}_w^{\alpha})
\end{equation}
for some $w\in\overline{\mathbb{D}}$ and $c\neq0$. Since $k_w^{\beta}=\widetilde{k}_w^{\beta}=1$, the above implies that
$$1+k_a=\overline{c}k_w^{\alpha}.$$
Equivalently,
\begin{equation}\label{uklad1}
\left\{\begin{array}{lcl}
\langle 1+k_a,1\rangle&=&\langle \overline{c}k_w^{\alpha},1\rangle\\
\langle 1+k_a,k_a\rangle&=&\langle \overline{c}k_w^{\alpha},k_a\rangle\\
\langle 1+k_a,k_{-a}\rangle&=&\langle \overline{c}k_w^{\alpha},k_{-a}\rangle
\end{array}\right..
\end{equation}
The first equation in \eqref{uklad1} implies that $c=2$. By the second equation in \eqref{uklad1},
$$1+\frac1{1-|a|^2}=\frac2{1-\overline{w}a},$$
and
$$w=\frac{a}{2-|a|^2}.$$
However, by the third equation,
$$1+\frac1{1+|a|^2}=\frac2{1+\overline{w}a},$$
and
$$w=\frac{a}{2+|a|^2}.$$
Clearly, $\frac{a}{2-|a|^2}=\frac{a}{2+|a|^2}$ if and only if $a=0$, and thus there is no $w\in\overline{\mathbb{D}}$ for which \eqref{row1} holds.

Assume then that
\begin{equation}\label{row2}
1\otimes(1+k_a)=c({k}_w^{\beta}\otimes \widetilde{k}_w^{\alpha})
\end{equation}
for some $w\in\overline{\mathbb{D}}$ and $c\neq0$. As before, we must have
$$1+k_a=\overline{c}\widetilde{k}_w^{\alpha},$$
which is equivalent to the following system of equations
\begin{equation}\label{uklad2}
\left\{\begin{array}{lcl}
\langle 1+k_a,1\rangle&=&\langle \overline{c}\widetilde{k}_w^{\alpha},1\rangle\\
\langle 1+k_a,k_a\rangle&=&\langle \overline{c}\widetilde{k}_w^{\alpha},k_a\rangle\\
\langle 1+k_a,k_{-a}\rangle&=&\langle \overline{c}\widetilde{k}_w^{\alpha},k_{-a}\rangle
\end{array}\right..
\end{equation}
Since $\langle\widetilde{k}_w^{\alpha},1\rangle=\alpha(w)/w$, the first equation in \eqref{uklad2} gives
\begin{equation}\label{ee1}
2=-\overline{c}\frac{a-w}{1-\overline{a}w}\frac{a+w}{1+\overline{a}w}.
\end{equation}
Similarly,
$$\langle \widetilde{k}_w^{\alpha},k_{a}\rangle=\frac{\alpha(w)}{w-a},\quad \langle \widetilde{k}_w^{\alpha},k_{-a}\rangle=\frac{\alpha(w)}{w+a},$$
and the second and third equation in \eqref{uklad2} yield
\begin{equation}\label{ee2}
1+\frac1{1-|a|^2}=\overline{c}\frac{w}{1-\overline{a}w}\frac{a+w}{1+\overline{a}w},
\end{equation}
and
\begin{equation}\label{ee3}
1+\frac1{1+|a|^2}=-\overline{c}\frac{w}{1+\overline{a}w}\frac{a-w}{1-\overline{a}w},
\end{equation}
respectively. By \eqref{ee1} and \eqref{ee2},
$$w=\frac{2-|a|^2}{\overline{a}}.$$
However, by \eqref{ee1} and \eqref{ee3},
$$w=\frac{2+|a|^2}{\overline{a}}.$$
As before, there is no $w\in\overline{\mathbb{D}}$ for which \eqref{row2} holds.
\end{exmp}

Note that in the above example every linear operator form $K_{\alpha}$ into $K_{\beta}$ is a rank-one operator and belongs to $\mathscr{T}(\alpha,\beta)$. This happens for all finite Blaschke products $\alpha$, $\beta$ of degree $m$, $n$, respectively, and such that $m=1$ or $n=1$. Indeed, if $\alpha$ and $\beta$ are finite Blaschke products of degree $m$ and $n$, respectively, then the dimension of $\mathscr{T}(\alpha,\beta)$ is $m+n-1$ (see \cite[Prop. 2.1]{blicharz2}). It follows that if $m=1$ or $n=1$, then
$$\dim \mathscr{T}(\alpha,\beta)=m+n-1=mn.$$
Since $mn$ is the dimension of $\mathcal{L}(\alpha,\beta)$, the space of all linear operators form $K_{\alpha}$ into $K_{\beta}$, we get $\mathscr{T}(\alpha,\beta)=\mathcal{L}(\alpha,\beta)$.

So if one of the spaces $K_{\alpha}$, $K_{\beta}$ is one-dimensional, then the set of rank-one asymmetric truncated Toeplitz operators is "quite big". Is it ever "big enough" to contain operators which are not scalar multiples of $\widetilde{k}_w^{\beta}\otimes {k}_w^{\alpha}$ or
$k_w^{\beta}\otimes \widetilde{k}_w^{\alpha}$? To answer this we need the following.

\begin{prop}\label{pepe}
Let $\alpha$ be a finite Blaschke product of degree $m>0$. Every $f\in K_{\alpha}$ is a scalar multiple of a reproducing kernel or a conjugate kernel if and only if $m\leq 2$.
\end{prop}
\begin{proof}
Let $\{v_{\eta_1}^{\alpha},\ldots,v_{\eta_m}^{\alpha}\}$ be the Clark basis of $K_{\alpha}$, corresponding to a fixed $\lambda\in\partial\mathbb{D}$.

We first prove that if every $f\in K_{\alpha}$ is a scalar multiple of a reproducing kernel or a conjugate kernel, then $m\leq 2$. To this end, we show that if $m>2$, then there exists $f\in K_{\alpha}$ that is neither a scalar multiple of a reproducing kernel nor a scalar multiple of a conjugate kernel.

Observe, that if
$$k_w^{\alpha}=c_1v_{\eta_1}^{\alpha}+\ldots+c_mv_{\eta_m}^{\alpha}$$
for some $w\in\overline{\mathbb{D}}$, then

\begin{displaymath}
\begin{split}
c_j&=\langle k_w^{\alpha},v_{\eta_j}^{\alpha}\rangle=\|k_{\eta_j}^{\alpha}\|^{-1}k_w^{\alpha}(\eta_j)\\
&=\frac1{\sqrt{|\alpha'(\eta_j)|}}\frac{1-\overline{\alpha(w)}\alpha(\eta_j)}{1-\overline{w}\eta_j}=\frac1{\sqrt{|\alpha'(\eta_j)|}}\frac{1-\overline{\alpha(w)}\alpha_{\lambda}}{1-\overline{w}\eta_j}.
\end{split}
\end{displaymath}
Hence, if $w\notin\{\eta_1,\ldots,\eta_m\}$, then $\alpha(w)\neq\alpha_{\lambda}$ and $c_j\neq0$ for all $1\leq j\leq m$. On the other hand, if $w=\eta_{j_0}$ for some $1\leq j_0\leq m$, then
\begin{equation}\label{cj}
c_j=\left\{\begin{array}{cl}
0&\mathrm{for}\ j\neq j_0\\
\sqrt{|\alpha'(\eta_{j_0})|}&\mathrm{for}\ j= j_0
\end{array}\right..
\end{equation}

Similarly, if
$$\widetilde{k}_w^{\alpha}=c_1v_{\eta_1}^{\alpha}+\ldots+c_mv_{\eta_m}^{\alpha}$$
for some $w\in\overline{\mathbb{D}}$, then
$$c_j=\langle \widetilde{k}_w^{\alpha},v_{\eta_j}^{\alpha}\rangle=\frac1{\sqrt{|\alpha'(\eta_j)|}}\frac{\alpha_{\lambda}-\alpha(w)}{\eta_j-w},$$
and $c_j\neq0$ for all $1\leq j\leq m$ if $w\notin\{\eta_1,\ldots,\eta_m\}$, or $c_j$'s satisfy \eqref{cj} if $w=\eta_{j_0}$ for some $1\leq j_0\leq m$.

Consider a linear combination
$$f=c_1v_{\eta_1}^{\alpha}+\ldots+c_mv_{\eta_m}^{\alpha}.$$
By the above, if $f$ is a scalar multiple of $k_w^{\alpha}$ or $\widetilde{k}_w^{\alpha}$, then all the coefficients $c_j$ are nonzero or there is precisely one nonzero coefficient $c_{j_0}$. If $m>2$ and $f$ is a linear combination as above with at least one coefficient equal to zero and at least two nonzero coefficients (for example $c_1=c_2=1$ and $c_j=0$ for $j>2$), then $f$ is neither a scalar multiple of a reproducing kernel nor a scalar multiple of a conjugate kernel.

To complete the proof we show that if $m\leq 2$, then every $f\in K_{\alpha}$ is a scalar multiple of a reproducing kernel or a conjugate kernel. This is clear for $m=1$.


Let $m=2$ and take $f=c_1v_{\eta_1}^{\alpha}+c_2v_{\eta_2}^{\alpha}$. If $c_1c_2=0$, then $f$ is a scalar multiple of a reproducing kernel. Assume that $c_1c_2\neq0$. We find $w\in\overline{\mathbb{D}}$ such that $f$ is a scalar multiple of $k_w^{\alpha}$ or $\widetilde{k}_w^{\alpha}$.

If $f=ck_w^{\alpha}$, then
\begin{displaymath}
\left\{	\begin{array}{c}
c_1=\frac{c}{\sqrt{|\alpha'(\eta_1)|}}\frac{1-\overline{\alpha(w)}\alpha_{\lambda}}{1-\overline{w}\eta_1}\\
c_2=\frac{c}{\sqrt{|\alpha'(\eta_2)|}}\frac{1-\overline{\alpha(w)}\alpha_{\lambda}}{1-\overline{w}\eta_2}
	\end{array}\right..
	\end{displaymath}
	From this
	$$\frac{1-\overline{w}\eta_1}{1-\overline{w}\eta_2}=\frac{\sqrt{|\alpha'(\eta_2)|}}{\sqrt{|\alpha'(\eta_1)|}}\frac{c_2}{c_1},$$
	and
	\begin{equation}\label{w1}
	w(\overline{\eta}_1-\overline{w}_0\overline{\eta}_2)=1-\overline{w}_0,
	\end{equation}
	where $w_0=\frac{\sqrt{|\alpha'(\eta_2)|}}{\sqrt{|\alpha'(\eta_1)|}}\frac{c_2}{c_1}$.
	
	On the other hand, if $f=c_1v_{\eta_1}^{\alpha}+c_2v_{\eta_2}^{\alpha}=c\widetilde{k}_w^{\alpha}$, then
$$\overline{c}_1\alpha_{\lambda}\overline{\eta}_1v_{\eta_1}^{\alpha}+\overline{c}_2\alpha_{\lambda}\overline{\eta}_2v_{\eta_2}^{\alpha}=\overline{c}k_w^{\alpha}.$$
As before,
\begin{displaymath}
\left\{	\begin{array}{c}
\overline{c}_1\alpha_{\lambda}\overline{\eta}_1=\frac{\overline{c}}{\sqrt{|\alpha'(\eta_1)|}}\frac{1-\overline{\alpha(w)}\alpha_{\lambda}}{1-\overline{w}\eta_1}\\
\overline{c}_2\alpha_{\lambda}\overline{\eta}_2=\frac{\overline{c}}{\sqrt{|\alpha'(\eta_2)|}}\frac{1-\overline{\alpha(w)}\alpha_{\lambda}}{1-\overline{w}\eta_2}
	\end{array}\right.
	\end{displaymath}
	and
	 $$\frac{\overline{\eta}_1-\overline{w}}{\overline{\eta}_2-\overline{w}}=\frac{\sqrt{|\alpha'(\eta_2)|}}{\sqrt{|\alpha'(\eta_1)|}}\frac{\overline{c}_2}{\overline{c}_1}=\overline{w}_0.$$
	From this
	\begin{equation}\label{w2}
	w(1-{w}_0)=\eta_1-{w}_0\eta_2.
	\end{equation}
	
Therefore, we must show that there exists $w\in\overline{\mathbb{D}}$ such that it satisfies \eqref{w1} (then $f$ is a scalar multiple of $k_w^{\alpha}$) or such that it satisfies \eqref{w2} (then $f$ is a scalar multiple of $\widetilde{k}_w^{\alpha}$). We consider three cases.
	
	\vspace{0.1cm}
	
	\textit{Case 1.} $1-w_0=0$.
	
	\vspace{0.1cm}
	
	In this case $w=0$ satisfies \eqref{w1}, $c_1\sqrt{|\alpha'(\eta_1)|}=c_2\sqrt{|\alpha'(\eta_2)|}$ and $f=ck_0^{\alpha}$ with
	$$c=\frac{c_1\sqrt{|\alpha'(\eta_1)|}}{1-\overline{\alpha(0)}\alpha_{\lambda}}=\frac{c_2\sqrt{|\alpha'(\eta_2)|}}{1-\overline{\alpha(0)}\alpha_{\lambda}}.$$
	
	\vspace{0.1cm}
	
	\textit{Case 2.} $\eta_1-w_0\eta_2=0$.
	
	\vspace{0.1cm}
	
	In this case $w=0$ satisfies \eqref{w2}, $c_1\eta_1\sqrt{|\alpha'(\eta_1)|}=c_2\eta_2\sqrt{|\alpha'(\eta_2)|}$ and $f=c\widetilde{k}_0^{\alpha}$ with
	$$c=\frac{c_1\eta_1\sqrt{|\alpha'(\eta_1)|}}{\alpha_{\lambda}-\alpha(0)}=\frac{c_2\eta_2\sqrt{|\alpha'(\eta_2)|}}{\alpha_{\lambda}-\alpha(0)}.$$
	
	\vspace{0.1cm}
	
	\textit{Case 3.} $1-w_0\neq0$ and $\eta_1-w_0\eta_2\neq0$.
	
	\vspace{0.1cm}
	
In this case \eqref{w1} has a solution given by
$$w_1=\frac{1-\overline{w}_0}{\overline{\eta}_1-\overline{w}_0\overline{\eta}_2}$$
and \eqref{w2} has a solution given by
$$w_2=\frac{{\eta}_1-{w}_0{\eta}_2}{1-{w}_0}.$$
Note that $w_2=1/\overline{w}_1$. Therefore, one of the numbers $w_1$, $w_2$ must be in $\overline{\mathbb{D}}$.

If $w_1\in\overline{\mathbb{D}}$, then $f=ck_{w_1}^{\alpha}$ with
	 $$c=c_1\sqrt{|\alpha'(\eta_1)|}\frac{1-\overline{w}_1\eta_1}{1-\overline{\alpha(w_1)}\alpha_{\lambda}}=c_2\sqrt{|\alpha'(\eta_2)|}\frac{1-\overline{w}_1\eta_2}{1-\overline{\alpha(w_1)}\alpha_{\lambda}}.$$

If $w_2\in\overline{\mathbb{D}}$, then $f=c\widetilde{k}_{w_2}^{\alpha}$ with
	$$c=c_1\sqrt{|\alpha'(\eta_1)|}\frac{\eta_1-w_2}{\alpha_{\lambda}-\alpha(w_2)}=c_2\sqrt{|\alpha'(\eta_2)|}\frac{\eta_2-w_2}{\alpha_{\lambda}-\alpha(w_2)}.$$
\end{proof}

\begin{cor}\label{mnkor}
Let $\alpha$ and $\beta$ be two finite Blaschke products of degree $m>0$ and $n>0$, respectively.
\begin{itemize}
\item[(a)] If $mn\leq2$, then every operator from $\mathscr{T}(\alpha,\beta)$ is a scalar multiple of $\widetilde{k}_w^{\beta}\otimes {k}_w^{\alpha}$ or ${k}_w^{\beta}\otimes \widetilde{k}_w^{\alpha}$ for some $w\in\overline{\mathbb{D}}$.
\item[(b)] If either $m=1$ and $n>2$, or $m>2$ and $n=1$, then there exists a rank-one operator from $\mathscr{T}(\alpha,\beta)$ that is neither a scalar multiple of $\widetilde{k}_w^{\beta}\otimes {k}_w^{\alpha}$ nor a scalar multiple of ${k}_w^{\beta}\otimes \widetilde{k}_w^{\alpha}$.
\end{itemize}
\end{cor}
\begin{proof}
$(\mathrm{a})$ Let $f\in K_{\alpha}$ and $g\in K_{\beta}$ be such that $g\otimes f\in\mathscr{T}(\alpha,\beta)$.

If $m=n=1$, then $f=c_1{k}_0^{\alpha}$, $g=c_2\widetilde{k}_0^{\beta}$ and $g\otimes f$ is a scalar multiple of $\widetilde{k}_0^{\beta}\otimes {k}_0^{\alpha}$.

If $m=2$, $n=1$, then, by Proposition \ref{pepe}, there exists $w\in\overline{\mathbb{D}}$ such that $f$ is a scalar multiple of ${k}_w^{\alpha}$ or $\widetilde{k}_w^{\alpha}$. Since here $K_{\beta}$ is one-dimensional, $g$ is always a scalar multiple of $\widetilde{k}_w^{\beta}$ and of ${k}_w^{\beta}$. Hence $g\otimes f$ is a scalar multiple of $\widetilde{k}_w^{\beta}\otimes {k}_w^{\alpha}$ or ${k}_w^{\beta}\otimes \widetilde{k}_w^{\alpha}$.

A similar reasoning can be used for $m=1$, $n=2$.

$(\mathrm{b})$ Let $m=1$ and $n>2$. By Proposition \ref{pepe}, there exists $g\in K_{\beta}$ that $g$ is neither a scalar multiple of a reproducing kernel nor a scalar multiple of a conjugate kernel. The desired operator is $A=g\otimes f$ with any $f\in K_{\alpha}$, $f\neq 0$.

The proof for $m>2$, $n=1$ is similar.
\end{proof}

What happens if both $m$ and $n$ are greater than one? It turns out that in that case every rank-one operator from $\mathscr{T}(\alpha,\beta)$ is a scalar multiple of $\widetilde{k}_w^{\beta}\otimes {k}_w^{\alpha}$ or ${k}_w^{\beta}\otimes \widetilde{k}_w^{\alpha}$ for some $w\in\overline{\mathbb{D}}$. To give the proof we need the following.

\begin{lem}\label{leon}
Let $\alpha$ and $\beta$ be two finite Blaschke products of degree $m>1$ and $n>1$, respectively. Let $f\in K_{\alpha}$, $g\in K_{\beta}$ be two nonzero functions such that $g\otimes f$ belongs to $\mathscr{T}(\alpha,\beta)$ and let $w\in\overline{\mathbb{D}}$. Then
\begin{itemize}
\item[(a)] $g$ is a scalar multiple of $k_w^{\beta}$ if and only if $f$ is a scalar multiple of $\widetilde{k}_w^{\alpha}$,
\item[(b)] $g$ is a scalar multiple of $\widetilde{k}_w^{\beta}$ if and only if $f$ is a scalar multiple of $k_w^{\alpha}$.
\end{itemize}
\end{lem}
\begin{proof}
Fix $\lambda_1,\lambda_2\in\partial\mathbb{D}$. Let $\{v_{\eta_1}^{\alpha},\ldots,v_{\eta_m}^{\alpha}\}$ be the Clark basis for $K_{\alpha}$ corresponding to $\lambda_1$ and let $\{v_{\zeta_1}^{\beta},\ldots,v_{\zeta_n}^{\beta}\}$ be the Clark basis for $K_{\beta}$ corresponding to $\lambda_2$. Assume moreover that the sets $\{\eta_1,\ldots,\eta_m\}$, $\{\zeta_1,\ldots,\zeta_n\}$ have precisely $l$ elements in common, these elements being $\eta_j=\zeta_j$, $j\leq l$.

Assume that $f\in K_{\alpha}$, $g\in K_{\beta}$ are such that $g\otimes f$ belongs to $\mathscr{T}(\alpha,\beta)$ and let $(r_{s,p})$ be the matrix representation  of $g\otimes f$ with respect to $\{v_{\eta_1}^{\alpha},\ldots,v_{\eta_m}^{\alpha}\}$ and $\{v_{\zeta_1}^{\beta},\ldots,v_{\zeta_n}^{\beta}\}$.

We only show that if $g=k_w^{\beta}$, then $f$ is a scalar multiple of $\widetilde{k}_w^{\alpha}$, and that if $g=\widetilde{k}_w^{\beta}$, then $f$ is a scalar multiple of $k_w^{\alpha}$. The proof in the general case is just a matter of multiplying by a constant. Moreover, the other implications then follow, because $g\otimes f$ belongs to $\mathscr{T}(\alpha,\beta)$ if and only if $f\otimes g$ belongs to $\mathscr{T}(\beta, \alpha)$.


We first give the proof for $l=0$, that is, when $\{\eta_1,\ldots,\eta_m\}$ and $\{\zeta_1,\ldots,\zeta_n\}$ have no elements in common.

By Theorem \ref{macierz3}, the matrix representation $(r_{s,p})$ satisfies \eqref{c1}. Since
$$r_{s,p}=\langle (g\otimes f)v_{\eta_p}^{\alpha},v_{\zeta_s}^{\beta}\rangle=\langle g,v_{\zeta_s}^{\beta}\rangle\langle v_{\eta_p}^{\alpha},f\rangle=\frac{g(\zeta_s)\overline{f(\eta_p)}}{\sqrt{|\beta'(\zeta_s)|}\sqrt{|\alpha'(\eta_p)|}},$$
the equation \eqref{c1} can be written as
$$g(\zeta_s)\overline{f(\eta_p)}=\frac{1-\overline{\eta}_1\zeta_s}{1-\overline{\eta}_p\zeta_s}g(\zeta_s)\overline{f(\eta_1)}-\frac{1-\overline{\eta}_1\zeta_1}{1-\overline{\eta}_p\zeta_s}g(\zeta_1)\overline{f(\eta_1)}+\frac{1-\overline{\eta}_p\zeta_1}{1-\overline{\eta}_p\zeta_s}g(\zeta_1)\overline{f(\eta_p)},$$
or
\begin{equation}\label{gefen}
[a_{s,p}g(\zeta_s)-a_{1,p}g(\zeta_1)]\overline{f(\eta_p)}=[a_{s,1}g(\zeta_s)-a_{1,1}g(\zeta_1)]\overline{f(\eta_1)},
\end{equation}
where
$$a_{s,p}=1-\overline{\eta}_p\zeta_s,\quad 1\leq p\leq m,\ 1\leq s\leq n.$$

Let $g=k_w^{\beta}$ with $w\in\overline{\mathbb{D}}$. We need to consider three cases.

\vspace{0.1cm}
	
\textit{Case 1.} $w=\zeta_{s_0}$ \textit{for some} $1\leq s_0\leq n$.
	
\vspace{0.1cm}

In this case $g(\zeta_{s_0})\neq0$ and $g(\zeta_s)=0$ for all $s\neq s_0$. If $s_0>1$, then \eqref{gefen} (with $s=s_0$) implies that
\begin{equation}\label{gzy}
a_{s_0,p}g(\zeta_{s_0})\overline{f(\eta_p)}=a_{s_0,1}g(\zeta_{s_0})\overline{f(\eta_1)}
\end{equation}
for all $1\leq p\leq m$. If $s_0=1$, then \eqref{gzy} also follows from \eqref{gefen} (with $s=n$ for example). Since $a_{s_0,p}\neq0$, we get
\begin{displaymath}
\begin{split}
f(\eta_p)&=\frac{\overline{a}_{s_0,1}}{\overline{a}_{s_0,p}}f(\eta_1)=\frac{1-\eta_1\overline{\zeta}_{s_0}}{1-\eta_p\overline{\zeta}_{s_0}}f(\eta_1)\\
&=\frac{1-\overline{\zeta}_{s_0}\eta_1}{1-\overline{\alpha(\zeta_{s_0})}\alpha_{\lambda_1}}f(\eta_1)k_{\zeta_{s_0}}^{\alpha}(\eta_p)=ck_{\zeta_{s_0}}^{\alpha}(\eta_p)	 
\end{split}
\end{displaymath}
for all $1\leq p\leq m$. From this,
$$\langle f,v_{\eta_p}^{\alpha}\rangle=\langle ck_{\zeta_{s_0}}^{\alpha},v_{\eta_p}^{\alpha}\rangle$$
for all $1\leq p\leq m$ and $f$ is a scalar multiple of $k_{\zeta_{s_0}}^{\alpha}$.
Since
\begin{equation}\label{kerk}
{k}_{w}^{\alpha}=\overline{\alpha(w)}{w}\widetilde{k}_{w}^{\alpha}
\end{equation}
for all $w\in\partial\mathbb{D}$, $f$ is a scalar multiple of $\widetilde{k}_{\zeta_{s_0}}^{\alpha}$.

\vspace{0.1cm}
	
\textit{Case 2.} $w=\eta_{p_0}$ \textit{for some} $1\leq p_0\leq m$.
	
\vspace{0.1cm}

Note that if $g=k_w^{\beta}$ with $w\in\overline{\mathbb{D}}\setminus\{\zeta_1,\ldots,\zeta_n\}$, then
\begin{equation}\label{zef}
\begin{split}
	 a_{s,p}g(\zeta_s)-a_{1,p}g(\zeta_1)&=(1-\overline{\eta}_{p}\zeta_s)\frac{1-\overline{\beta(w)}\beta_{\lambda_2}}{1-\overline{w}\zeta_s}-(1-\overline{\eta}_{p}\zeta_1)\frac{1-\overline{\beta(w)}\beta_{\lambda_2}}{1-\overline{w}\zeta_1}\\
	&=\frac{(1-\overline{\beta(w)}\beta_{\lambda_2})(\overline{w}-\overline{\eta}_p)(\zeta_s-\zeta_1)}{(1-\overline{w}\zeta_s)(1-\overline{w}\zeta_1)}.
	\end{split}
\end{equation}
It follows that in this case
\begin{equation}\label{zz1}
a_{n,p_0}g(\zeta_n)-a_{1,p_0}g(\zeta_1)=0,
\end{equation}
and
\begin{equation}\label{zz2}
a_{n,p}g(\zeta_n)-a_{1,p}g(\zeta_1)\neq0
\end{equation}
for all $p\neq p_0$.

If $p_0=1$, then \eqref{gefen} (with $s=n$) and \eqref{zz1} imply that
$$[a_{n,p}g(\zeta_n)-a_{1,p}g(\zeta_1)]\overline{f(\eta_p)}=[a_{n,1}g(\zeta_n)-a_{1,1}g(\zeta_1)]\overline{f(\eta_1)}=0.$$
By \eqref{zz2}, $f(\eta_p)=0$ for all $p\neq p_0$. Clearly, $f$ is a scalar multiple of $k_{\eta_{p_0}}^{\alpha}$ and so a scalar multiple of $\widetilde{k}_{\eta_{p_0}}^{\alpha}$ by \eqref{kerk}.

If $p_0>1$, then \eqref{gefen} (with $p=p_0$ and $s=n$) and \eqref{zz1} give
$$0=[a_{n,p_0}g(\zeta_n)-a_{1,p_0}g(\zeta_1)]\overline{f(\eta_{p_0})}=[a_{n,1}g(\zeta_n)-a_{1,1}g(\zeta_1)]\overline{f(\eta_1)}.$$
By \eqref{zz2}, $f(\eta_1)=0$. From this and \eqref{gefen} again (with $s=n$),
$$[a_{n,p}g(\zeta_n)-a_{1,p}g(\zeta_1)]\overline{f(\eta_p)}=0$$
for all $p$ and therefore $f(\eta_p)=0$ for all $p\neq p_0$. As before, $f$ is a scalar multiple of $\widetilde{k}_{\eta_{p_0}}^{\alpha}$.

\vspace{0.1cm}
	
\textit{Case 3.} $w\notin\{\eta_1,\ldots,\eta_m\}\cup\{\zeta_1,\ldots,\zeta_n\}$.
	
\vspace{0.1cm}

In this case \eqref{gefen} (with $s=n$) gives
$$f(\eta_p)=\frac{\overline{a}_{n,1}\overline{g(\zeta_n)}-\overline{a}_{1,1}\overline{g(\zeta_1)}}{\overline{a}_{n,p}\overline{g(\zeta_n)}-\overline{a}_{1,p}\overline{g(\zeta_1)}}f(\eta_1)$$
for all $p>1$. By \eqref{zef},

\begin{displaymath}
\begin{split}
f(\eta_p)&=\frac{w-\eta_1}{w-\eta_p}f(\eta_1)=\\
&=\frac{\eta_1-w}{\alpha_{\lambda_1}-\alpha(w)}f(\eta_1)\widetilde{k}_{w}^{\alpha}(\eta_p)=c\widetilde{k}_{w}^{\alpha}(\eta_p).	
\end{split}
\end{displaymath}
Hence $f$ is a scalar multiple of $\widetilde{k}_{w}^{\alpha}$.

This completes the proof of the first implication (for $l=0$).

Now let $g=\widetilde{k}_{w}^{\beta}$. If $w\in\partial\mathbb{D}$, then $g$ is a scalar multiple of $k_{w}^{\beta}$ by \eqref{kerk} and $f$ is a scalar multiple of $\widetilde{k}_{w}^{\alpha}$ by the first part of the proof. Assume that $w\in\mathbb{D}$. Then
\begin{equation}\label{zef2}
	a_{s,p}g(\zeta_s)-a_{1,p}g(\zeta_1)=\frac{(\beta_{\lambda_2}-\beta(w))(1-\overline{\eta}_pw)(\zeta_1-\zeta_s)}{(\zeta_s-w)(\zeta_1-w)}.
\end{equation}
By \eqref{gefen} (with $s=n$),
\begin{displaymath}
\begin{split}
f(\eta_p)&=\frac{1-\overline{w}\eta_1}{1-\overline{w}\eta_p}f(\eta_1)\\
&=\frac{1-\overline{w}\eta_1}{1-\overline{\alpha(w)}\alpha_{\lambda_1}}f(\eta_1)k_{w}^{\alpha}(\eta_p)=ck_{w}^{\alpha}(\eta_p)	
\end{split}
\end{displaymath}
for all $1\leq p\leq m$, and $f$ is a scalar multiple of $k_{w}^{\alpha}$.

This completes the proof of the second implication (for $l=0$).

We now give the proof for $l>0$. In this case the matrix representation $(r_{s,p})$ of $g\otimes f$ must satisfy \eqref{c2} and \eqref{c3}. Instead of \eqref{gefen} we obtain
 \begin{equation}\label{gefen2}
a_{s,p}g(\zeta_s)\overline{f(\eta_p)}=a_{1,p}g(\zeta_1)\overline{f(\eta_p)}-a_{1,s}g(\zeta_1)\overline{f(\eta_s)},
\end{equation}
for all $p$, $s$ such that $1\leq p\leq m$, $1\leq s\leq l$, $s\neq p$, and
\begin{equation}\label{gefen3}
a_{s,p}g(\zeta_s)\overline{f(\eta_p)}=a_{s,1}g(\zeta_s)\overline{f(\eta_1)}+a_{1,p}g(\zeta_1)\overline{f(\eta_p)},
\end{equation}
for all $p$, $s$ such that $1\leq p\leq m$, $s>l$.

Let $g=k_w^{\beta}$ with $w\in\overline{\mathbb{D}}$. As before, we have three cases to consider.

\vspace{0.1cm}
	
\textit{Case 1.} $w=\zeta_{s_0}$ \textit{for some} $1\leq s_0\leq n$.
	
\vspace{0.1cm}

In this case $g(\zeta_{s_0})\neq0$ and $g(\zeta_s)=0$ for all $s\neq s_0$.

If $s_0=1=l$, then \eqref{gefen3} (with $s=n>1$) gives
$$0=a_{n,p}g(\zeta_{n})\overline{f(\eta_p)}=a_{n,1}g(\zeta_{n})\overline{f(\eta_1)}+a_{1,p}g(\zeta_{1})\overline{f(\eta_p)}=a_{1,p}g(\zeta_{1})\overline{f(\eta_p)}$$
for all $1\leq p\leq m$ and so $f(\eta_p)=0$ for all $p>1$. Hence $f$ is a scalar multiple of ${k}_{\eta_{s_0}}^{\alpha}={k}_{\zeta_{s_0}}^{\alpha}$ (since here $\eta_{s_0}=\zeta_{s_0}$) and a scalar multiple of $\widetilde{k}_{\zeta_{s_0}}^{\alpha}$ by \eqref{kerk}.

If $s_0=1<l$, then \eqref{gefen2} (with $p=1$) implies that
$$0=a_{s,1}g(\zeta_{s})\overline{f(\eta_1)}=-a_{1,s}g(\zeta_{1})\overline{f(\eta_s)}$$
for every $s\neq s_0=1$, $s\leq l$. Form this, $f(\eta_s)=0$ for all $1<s\leq l$. If $m=l$, then this implies that $f$ is a scalar multiple of ${k}_{\zeta_{s_0}}^{\alpha}$ and as above, a scalar multiple of $\widetilde{k}_{\zeta_{s_0}}^{\alpha}$. If $m>l$, then by \eqref{gefen2} (with $s=l$, $p>l$),
$$0=a_{l,p}g(\zeta_{l})\overline{f(\eta_p)}=a_{1,p}g(\zeta_{1})\overline{f(\eta_p)}$$
and $f(\eta_p)=0$ for all $p>l$, which also implies that $f$ is a scalar multiple of $\widetilde{k}_{\zeta_{s_0}}^{\alpha}$.

If $1<s_0\leq l$, then \eqref{gefen2} (with $s=s_0$) implies that
$$a_{s_0,p}g(\zeta_{s_0})\overline{f(\eta_p)}=0$$
and $f(\eta_p)=0$ for all $p\neq s_0$. Hence $f$ is a scalar multiple of $\widetilde{k}_{\zeta_{s_0}}^{\alpha}$.

If $1\leq l<s_0$, then \eqref{gefen3} (with $s=s_0$) gives
$$a_{s_0,p}g(\zeta_{s_0})\overline{f(\eta_p)}=a_{s_0,1}g(\zeta_{s_0})\overline{f(\eta_1)}$$
and
\begin{displaymath}
\begin{split}
f(\eta_p)&=\frac{\overline{a}_{s_0,1}}{\overline{a}_{s_0,p}}f(\eta_1)=\frac{1-\overline{\zeta_{s_0}}\eta_1}{1-\overline{\zeta_{s_0}}\eta_p}f(\eta_1)\\
&=\frac{1-\overline{\zeta_{s_0}}\eta_1}{1-\overline{\alpha(\zeta_{s_0})}\alpha_{\lambda_1}}f(\eta_1)k_{\zeta_{s_0}}^{\alpha}(\eta_p)=ck_{\zeta_{s_0}}^{\alpha}(\eta_p)
\end{split}
\end{displaymath}
for all $1\leq p\leq m$, which again implies that $f$ is a scalar multiple of $\widetilde{k}_{\zeta_{s_0}}^{\alpha}$.

\vspace{0.1cm}
	
\textit{Case 2.} $w=\eta_{p_0}$ \textit{for some} $l<p_0\leq m$.
	
\vspace{0.1cm}

In this case \eqref{zef} implies that
$$a_{n,p_0}g(\zeta_n)-a_{1,p_0}g(\zeta_1)=0,$$
and
$$a_{n,p}g(\zeta_n)-a_{1,p}g(\zeta_1)\neq0$$
for all $p\neq p_0$.

If $l=n$, then \eqref{gefen2} (with $p=p_0$, $s=n$) gives
$$0=[a_{n,p_0}g(\zeta_n)-a_{1,p_0}g(\zeta_1)]\overline{f(\eta_{p_0})}=-a_{1,n}g(\zeta_1)\overline{f(\eta_{n})}.$$
Since $g(\zeta_1)\neq 0$, we have $f(\eta_{n})=0$. By \eqref{gefen2} again (with $s=n$),
\begin{equation}\label{zef3}
[a_{n,p}g(\zeta_n)-a_{1,p}g(\zeta_1)]\overline{f(\eta_{p})}=-a_{1,n}g(\zeta_1)\overline{f(\eta_{n})}=0
\end{equation}
for $p\neq n$ and so $f(\eta_{p})=0$ for all $p\neq p_0$. Hence $f$ is a scalar multiple of $\widetilde{k}_{\eta_{p_0}}^{\alpha}$.

If $l<n$, then \eqref{gefen3} (with $p=p_0$, $s=n$) implies that $f(\eta_{1})=0$ and that \eqref{zef3} holds for all $1\leq p\leq m$. Hence $f(\eta_{p})=0$ for all $p\neq p_0$ and $f$ is a scalar multiple of $\widetilde{k}_{\eta_{p_0}}^{\alpha}$.

\vspace{0.1cm}
	
\textit{Case 3.} $w\notin\{\eta_1,\ldots,\eta_m\}\cup\{\zeta_1,\ldots,\zeta_n\}$.

\vspace{0.1cm}

If $l=n$, then \eqref{gefen2} (with $p=1$, $s=n$) implies that

\begin{equation}\label{fefen}
a_{n,1}g(\zeta_n)\overline{f(\eta_{1})}=-a_{1,n}g(\zeta_1)\overline{f(\eta_{n})}.
\end{equation}
From this
\begin{displaymath}
\begin{split}
f(\eta_n)&=-\frac{(1-\eta_1\overline{\zeta}_{n})(1-w\overline{\zeta}_1)}{(1-\eta_n\overline{\zeta}_{1})(1-w\overline{\zeta}_n)}f(\eta_1)\\
&=-\frac{(\zeta_n-\eta_{1})(\zeta_1-w)}{(\zeta_1-\eta_{n})(\zeta_n-w)}f(\eta_1)=\frac{\eta_1-w}{\eta_n-w}f(\eta_1).
\end{split}
\end{displaymath}

Moreover, \eqref{gefen2} (with $s=n$) together with \eqref{fefen} gives
\begin{equation}\label{zz3}
[a_{n,p}g(\zeta_n)-a_{1,p}g(\zeta_1)]\overline{f(\eta_{p})}=a_{n,1}g(\zeta_n)\overline{f(\eta_{1})}
\end{equation}
and, by \eqref{zef},
\begin{displaymath}
\begin{split}
f(\eta_{p})&=\frac{\overline{a}_{n,1}\overline{g(\zeta_n)}}{\overline{a}_{n,p}\overline{g(\zeta_n)}-\overline{a}_{1,p}\overline{g(\zeta_1)}}f(\eta_{1})\\
&=\frac{(1-\eta_1\overline{\zeta}_{n})(1-w\overline{\zeta}_1)}{(\bar{\zeta}_1-\bar{\zeta}_n)(\eta_{p}-w)}f(\eta_1)=\frac{\eta_1-w}{\eta_p-w}f(\eta_1)
\end{split}
\end{displaymath}
for all $p\neq n$. Hence
\begin{displaymath}
f(\eta_{p})=\frac{\eta_1-w}{\eta_p-w}f(\eta_1)=\frac{\eta_1-w}{\alpha_{\lambda_1}-\alpha(w)}f(\eta_1)\widetilde{k}_w^{\alpha}(\eta_p)=c\widetilde{k}_w^{\alpha}(\eta_p)
\end{displaymath}
for all $p$, and $f$ is a scalar multiple of $\widetilde{k}_{w}^{\alpha}$.

If $l<n$, then it follows form \eqref{gefen3} (with $s=n$) that \eqref{zz3} holds for all $p$. As above
\begin{displaymath}
f(\eta_{p})=\frac{\eta_1-w}{\eta_p-w}f(\eta_1)=c\widetilde{k}_w^{\alpha}(\eta_p)
\end{displaymath}
 and $f$ is a scalar multiple of $\widetilde{k}_{w}^{\alpha}$.

 This completes the proof of the first implication (for $l>0$).

Now let $g=\widetilde{k}_{w}^{\beta}$. As in the proof for $l=0$, we can assume that $w\in\mathbb{D}$. Repeating the argument form Case 3 above one can show that
\begin{displaymath}
\begin{split}
f(\eta_p)&=\frac{1-\overline{w}\eta_1}{1-\overline{w}\eta_p}f(\eta_1)\\
&=\frac{1-\overline{w}\eta_1}{1-\overline{\alpha(w)}\alpha_{\lambda_1}}f(\eta_1)k_w^{\alpha}(\eta_p)=ck_w^{\alpha}(\eta_p)
\end{split}
\end{displaymath}
for all $p$. Hence $f$ is a scalar multiple of $k_{w}^{\alpha}$.

This completes the proof of the second implication (for $l>0$) and the proof of the theorem.
\end{proof}

\begin{thm}
Let $\alpha$ and $\beta$ be two finite Blaschke products of degree $m>0$ and $n>0$, respectively. The only rank-one operators in $\mathscr{T}(\alpha,\beta)$ are the nonzero scalar multiples of the operators $\widetilde{k}_w^{\beta}\otimes k_w^{\alpha}$ and $k_w^{\beta}\otimes \widetilde{k}_w^{\alpha}$, $w\in\overline{\mathbb{D}}$, if and only if either $mn\leq2$, or $m>1$ and $n>1$.
\end{thm}
\begin{proof}
By Corollary \ref{mnkor}, if the only rank-one operators in $\mathscr{T}(\alpha,\beta)$ are the nonzero scalar multiples of $\widetilde{k}_w^{\beta}\otimes k_w^{\alpha}$ and $k_w^{\beta}\otimes \widetilde{k}_w^{\alpha}$, $w\in\overline{\mathbb{D}}$, then either $mn\leq2$, or $m>1$ and $n>1$.

To complete the proof we only need to show that if $m>1$ and $n>1$, then every rank-one operator from $\mathscr{T}(\alpha,\beta)$ is a scalar multiple of $\widetilde{k}_w^{\beta}\otimes k_w^{\alpha}$ or $k_w^{\beta}\otimes \widetilde{k}_w^{\alpha}$ for some $w\in\overline{\mathbb{D}}$. The proof is similar to that of \cite[Thm. 5.1(c)]{s}. Sarason's proof is based on the fact that every truncated Topelitz operator is complex symmetric. Here one must use Lemma \ref{leon} instead. The details are left to the reader.
\end{proof}

\end{document}